\documentclass[12pt, reqno, a4paper]{amsart}
\usepackage[a4paper, total={6in, 10in}]{geometry}
\usepackage[T1]{fontenc}
\usepackage[utf8]{inputenc}

\usepackage{todonotes}
\usepackage[all]{xy}

\usepackage{amssymb}   % Ekstra symboler
\usepackage{amsthm}      % Teoremaktige omgivelser
\usepackage{thmtools}     % Teoremaktige omgivelser
\usepackage{mathtools}   % Fonter og omgivelser for matematikk
\usepackage{mathrsfs}     % Scriptfont med \mathscr{}
\usepackage{physics}       % Snarveier for partielle deriverte og differensialer 

\usepackage[backend=biber,style=numeric-comp, sorting=nyt, maxnames=99]{biblatex}
\usepackage[hidelinks, hypertexnames=false]{hyperref}
\usepackage[nameinlink, noabbrev, capitalize]{cleveref}
\usepackage{bookmark}

\declaretheorem[style = plain, numberwithin=section]{theorem}

\declaretheorem[style = plain,      sibling = theorem]{lemma}

\usepackage[scaled]{beramono}  % Bera Mono (Bitstream Vera Sans Mono)
\usepackage{listings}
\usepackage{xcolor}

\definecolor{codegreen}{rgb}{0,0.6,0}
\definecolor{codegray}{rgb}{0.5,0.5,0.5}
\definecolor{backcolour}{rgb}{0.95,0.95,0.92}
\definecolor{codegreen}{rgb}{0,  0.6,  0.35}

\lstdefinestyle{mystyle}{
	backgroundcolor=\color{backcolour},   
	commentstyle=\color{codegreen},
	keywordstyle=\color{blue},
	numberstyle=\tiny\color{codegray},
	stringstyle=\color{codegreen},
	basicstyle=\ttfamily\footnotesize,
	breakatwhitespace=false,         
	breaklines=true,                 
	captionpos=b,                    
	keepspaces=true,                 
	numbers=left,                    
	numbersep=5pt,                  
	showspaces=false,                
	showstringspaces=false,
	showtabs=false,                  
	tabsize=2
}
\usepackage[shortlabels]{enumitem}
\setenumerate{label=(\roman{enumi})}

\newcommand{\bbN}{\mathbb{N}}   % Naturlige tall
\newcommand{\bbC}{\mathbb{C}}   % Komplekse tall
\newcommand{\bbF}{\mathbb{F}}
\newcommand{\bbT}{\mathbb{T}}

\newcommand{\cU}{\mathcal{U}}

\usepackage{graphicx}                                                         % Bilder 

\title{Commutators in selfless C*-algebras}

\author{Alexander G. Ravnanger}
\address{ \noindent \newline 
	Department of Mathematical Sciences \newline 
	University of Copenhagen \newline
	Universitetsparken 5 \newline 
	2100 Copenhagen}
\email[]{agr@math.ku.dk}

\begin{document}
	\begin{abstract}
		We show that if \((A, \tau)\) is a selfless C*-probability space with \(\tau\) being a faithful trace, then commutators are dense in \(\ker \tau\). 
	\end{abstract}
	
	\maketitle

\section{Introduction}
	
	The question of when an operator on a Hilbert space can be written as a commutator first receieved serious interest because of the Heisenberg uncertainty principle in physics. The Wielandt-Wintner theorem, see \cite{wielandt,wintner}, states that the identity operator on a Hilbert space cannot be expressed as the commutator of two bounded operators. Wielandt already pointed out in his paper that his proof is valid in any normed unital ring, thus opening the door to an abstract analysis of the problem. A popular variation of the problem is to write elements vanishing on every trace in a C*-algebra as a sum of a bounded number of commutators. Recently, Ciu and Ma deviced a method to reduce sums of commutators to single commutators using matrix algebra. We refer the reader to their paper \cite{cui_ma} and the references therein for an overview of the area. Their result, building on the results of \cite{marcoux}, show that for a broad class of simple C*-algebras, including UHF-algebras, irrational rotation algebras and Cuntz algebras, the set of elements vanishing on all traces coincides with the set of single commutators. This points to a common theme in the literature on the topic: Most avaliable results require some form of abundance of projections, since this allows the reduction of the problem to (approximately) solving matrix equations with respect to a partition of the unit by projections. The projectionless case has been left largely unexplored. 
	
	Alain Valette posed on MathOverflow \cite{MO} the question if every element deleted by the unique trace on the reduced group C*-algebra of a free group is close to a sum of commutators. In fact, much more is true: Any element in a C*-algebra vanishing on all bounded traces is close to a sum of commutators; this follows from a Hahn-Banach argument akin to one in \cite{pop}. Nevertheless, Jesse Peterson provided an elegant answer for the specific situation using Powers' averaging technique. Prompted by the author with the question of whether the representative of a group generator is close to a single commutator in \(C_\lambda^*(\bbF_2)\), Claude Opus 5.5 generated a variation of Peterson's argument using Abel's summation trick to show that this is indeed the case. Inspired by this trick, we show in this note that if \((A, \tau)\) is a C*-probability space with \(\tau\) being a faithful trace, then commutators are dense in \(\ker \tau\). This constitutes a generalization of a result by Robert in his seminal paper \cite{robert}, where he also requires that \(A\) have real rank zero. In particular, it applies to the stably projectionless C*-algebras \(C_\lambda^*(\bbF_n)\) for \(n \geq 2\), which were shown to be selfless in the landmark paper \cite{AGKEP}, and the Jiang-Su algebra, see \cite[Theorem 5.2]{robert}. It remains open if the set of commutators is closed in both cases.
	
	I wish to thank Leonel Robert and Mikael Rørdam for reading drafts of this note. 
 	
\section{Approximation by commutators}

	A \textit{C*-probability space} is a pair \((A, \rho)\) consisting of a unital C*-algebra \(A\) and a state \(\rho\) on \(A\). A C*-probability space \((A, \rho)\) is \textit{selfless} if for some C*-probability space \((C, \kappa)\) such that \(C \neq \bbC\) and \(\kappa\) has faithful GNS-representation, there is an ultrafilter \(\cU\) such that the diagonal embedding \(\Delta \colon A \to A^\cU\) factors through the first factor embedding \(\iota \colon (A, \rho) \mathop{*} (C, \kappa)\), i.e., there is embedding \(\varphi \colon (A, \tau) \mathop{*} (C, \kappa) \to A^\cU\) such that \(\Delta = \varphi \iota\). In fact, the pair \((C, \kappa)\) may be assumed to be any one of a series of well-known C*-probability spaces, see \cite[Theorem 2.6]{robert}. Selflessness is a strong regularity property for C*-probability spaces bundling together many well-studied regularity properties for simple C*-algebras. Importantly, Robert showed in \cite[Theorem 3.1]{robert} that if \((A, \rho)\) is a selfless C*-probability space, then \(A\) is simple and satisfies the following dichotomy: If \(\rho\) is not a trace, then \(A\) is purely infinite, and if \(\rho\) is a trace, then \(A\) has stable rank one and strict comparison of positive elements, and \(\rho\) is the unique 2-quasitrace on \(A\). We are interested in the tracial case since if \(A\) is purely infinite, then any non-central element is a single commutator by \cite{pop} and \cite{cui_ma}. 
	
	We will use the following trick. 
	
	\begin{lemma}
		\label{lemma_abel}
		Let \(A\) be a unital C*-algebra and fix a unitary element \(u \in A\). Let \(x \in A\) and \(n \in \bbN\) and set \[y = \sum_{k=0}^{n-1} \left( 1 - \frac{k}{n} \right) u^k x u^{-k-1}.\] Then 
		\[x - \frac{1}{n} \sum_{k=1}^n u^k x u^{-k} = [y, u] = z - uzu^*,\]
		where \(z = yu\). 
	\end{lemma}
	
	\begin{proof}
		By using Abel summation on the scalar and non-scalar part of the terms in the first sum, we have
			\[ [y,u] = \sum_{k=0}^{n-1} \left(1 - \frac{k}{n} \right)(u^k x u^{-k} - u^{k+1}x u^{-(k+1)}) = x - \frac{1}{n} \sum_{k=1}^n u^k x u^{-k}.\] The second equality is trivial. 
	\end{proof}
	
	\begin{theorem}
		Let \((A, \tau)\) be a selfless C*-probability space such that \(\tau\) is a faithful trace. Then commutators are dense in \(\ker \tau\). In fact, any element in \(\ker \tau\) may be arbitrarily well approximated by a commutator of one element with a unitary. 
	\end{theorem}
	
	\begin{proof}
		Fix an element \(x \in \ker \tau\) and \(\varepsilon > 0\). Let \(\mu\) denote the state on \(C(\bbT)\) corresponding to the normalized Lebesgue measure on \(\bbT\). Since \((A, \tau)\) is selfless, there is an embedding \(\varphi \colon (A, \tau) \mathop{*} (C(\bbT), \mu) \to A^\cU\) into an ultrapower of \(A\) such that \(\Delta = \varphi \iota,\) following the same notation as above. We have \((\tau \mathop{*} \mu)(\iota(x)) = \tau(x) = 0\). Let \(u\) denote the image of the identity function on \(\bbT\) in \(A \mathop{*} C(\bbT)\). Then, cf. \cite[Proposition 3.2]{dykema}, there is an \(n \in \bbN\) such that \[\norm{\frac{1}{n} \sum_{k=1}^n u^k \iota(x) u^{-k}} < \varepsilon.\] By \cref{lemma_abel},
		\[ \norm{\iota(x) - [y,u]} =\norm{\frac{1}{n} \sum_{k=1}^n u^k \iota(x) u^{-k}} < \varepsilon,\]
		where 
		 \[y = \sum_{k=0}^{n-1} \left( 1-\frac{k}{n} \right) u^k \iota(x) u^{-k-1}.\] Hence, \(\norm{\Delta(x) - [\varphi(y), \varphi(u)]} <\varepsilon\) in \(A^\cU\), which means that there is a unitary \(v \in A\) and element \(z \in A\) such that \(\norm{x - [z,v]} < \varepsilon.\)
	\end{proof} 
	
	%Let \(\sigma\) denote the trace on \(C_\lambda^*(\bbF_\infty)\). Since \((A, \tau)\) is selfless, there is an embedding \(\varphi \colon (A, \tau) \mathop{*} (C_\lambda^*(\bbF_\infty), \sigma) \to A^\cU\) into an ultrapower of \(A\) such that \(\Delta = \varphi \iota,\) following the same notation as above. We have \((\tau \mathop{*} \sigma)(\iota(x)) = \tau(x) = 0\), and so by the proof of \cite[Proposition 3.2]{dykema}, there is a unitary \(u \in A \mathop{*} C_\lambda^*(\bbF_\infty)\) and an \(n \in \bbN\) such that \[\norm{\frac{1}{n} \sum_{k=1}^n u^k \iota(x) u^{-k}} < \varepsilon.\]

\printbibliography
	
\end{document}